\documentclass{amsart}

\usepackage{amsfonts,amssymb}
\usepackage{enumerate,graphicx}
\usepackage{subcaption}
\usepackage[square,numbers]{natbib}
\usepackage{pstricks-add}
\usepackage{MnSymbol}
\usepackage{mathrsfs,bbm}

\newtheorem{theorem}{Theorem}[section]

\newtheorem{proposition}[theorem]{Proposition}
\newtheorem{corollary}[theorem]{Corollary}
\newtheorem{lemma}[theorem]{Lemma}

\theoremstyle{definition}

\newtheorem{example}[theorem]{Example}

\newtheorem{remark}[theorem]{Remark}

\begin{document}

\title{Topological Degree of Shift Spaces on Monoids}


\author{Jung-Chao Ban}
\address[Jung-Chao Ban]{Department of Applied Mathematics, National Dong Hwa University, Hualien 97401, Taiwan, ROC.}
\email{jcban@gms.ndhu.edu.tw}

\author[Chih-Hung Chang]{Chih-Hung Chang*}
\thanks{*Author to whom any correspondence should be addressed.} 
\author{Nai-Zhu Huang}
\address[Chih-Hung Chang and Nai-Zhu Huang]{Department of Applied Mathematics, National University of Kaohsiung, Kaohsiung 81148, Taiwan, ROC.}
\email{chchang@nuk.edu.tw; naizhu7@gmail.com}

\date{June 28, 2018}

\baselineskip=1.5\baselineskip

\begin{abstract}
This paper considers the topological degree of $G$-shifts of finite type for the case where $G$ is a nonabelian monoid. Whenever the Cayley graph of $G$ has a finite representation and the relationships among the generators of $G$ are determined by a matrix $A$, the coefficients of the characteristic polynomial of $A$ are revealed as the number of children of the graph. After introducing an algorithm for the computation of the degree, the degree spectrum, which is finite, relates to a collection of matrices in which the sum of each row of every matrix is bounded by the number of children of the graph. Furthermore, the algorithm extends to $G$ of finite free-followers.
\end{abstract}

\maketitle

\section{Introduction}

Let $\mathcal{A}$ be a finite alphabet. Given $d \in \mathbb{N}$, a \emph{configuration} is a function from $\mathbb{Z}^d$ to $\mathcal{A}$, and a pattern is a function from a finite subset of $\mathbb{Z}^d$ to $\mathcal{A}$. A subset $X \subseteq \mathcal{A}^{\mathbb{Z}^d}$ is called a \emph{shift space} if $X$, denoted by $X = \mathsf{X}_{\mathcal{F}}$, consists of configurations which avoid patterns from some set $\mathcal{F}$ of patterns. A shift space is called a \emph{shift of finite type} (SFT) if $\mathcal{F}$ is a finite set; $\mathbb{Z}^d$ acts on $X$ by translation of configurations making $X$ a symbolic dynamical system. One of the many motivations in studying symbolic dynamical systems is it helps for the investigation of hyperbolic topological dynamical systems. The interested reader can consult standard literature such as \cite{Bow-1975, Rue-1978}.

While almost all properties of SFTs are decidable for $d=1$ (cf.~\cite{LM-1995}), the investigation of SFTs is rife for $d \geq 2$ since many undecidability issues company with it. It is even undecidable if an SFT is nonempty \cite{Berger-MAMS1966}. Different kinds of mixing properties have been introduced for examining the existence and denseness of periodic configurations \cite{BPS-TAMS2010}. A straightforward generalization is considering SFTs on $G$ which is associated with some algebraic structure. Whenever $G$ is a free monoid, it has been demonstrated that many such issues do not arise. For instance, the conjugacy between two irreducible $G$-SFTs is decidable \cite{AB-TCS2012}; furthermore, the nonemptiness, extensibility, and the existence of periodic configurations are decidable for $G$-SFTs \cite{BC-JMP2017, BC-TAMS2017}. Aside from the qualitative behavior of $G$-SFTs, the phenomena from the computational perspective are also fruitful \cite{BC-2017a, BC-2017b, PS-TCS2018, Piantadosi-DCDS2008}.

For the case where $G = \mathbb{Z}^1$, the topological entropy of an SFT relates to the spectral radius of an integral matrix, and the entropy spectrum (i.e., the set of entropies) of SFTs is the set of logarithms of Perron numbers \cite{Lind-ETDS1984, LM-1995}. When $G = \mathbb{Z}^d$ for $d \geq 2$, the entropy of an SFT is a right recursively enumerable number which may not be algebraic and is not computable \cite{HM-AoM2010, MP-ETDS2013, PS-TAMS2015}. However, the story is quite different when $G$ is a free monoid.

Suppose that $G$ is a finitely generated free monoid. Let $\Sigma$ be a finite set which generates $G$. An element $g \in G$ is called an \emph{$i$-word} provided the minimal expresion of $g$ is $g = g_1 g_2 \cdots g_i$ for some $g_1, \ldots, g_i \in \Sigma$. For $n \in \mathbb{N}$, let $\Gamma_n$ denote the set of $n$-blocks in a $G$-SFT; an $n$-block is a pattern whose support consists of all $i$-words in $G$ for $i \leq n$. Suppose the cardinality of $\Gamma_n$ behaves approximately like $c(n) \lambda^{\kappa^n}$ for some $\lambda, \kappa \in \mathbb{R}$ and $c(n) = o(\kappa^n)$. It follows that such a $G$-SFT carries the topological entropy $\ln \lambda$ if $\kappa = d$ is the number of generators of $G$ and $0$ otherwise, and the topological degree (defined in \eqref{eq:defn-degree}) is $\ln \kappa$ \cite{BC-N2017}. The degree spectrum (i.e., the set of degrees) of $G$-SFTs relates to the set of Perron numbers \cite{BC-2017}, and Petersen and Salama reveal an algorithm to estimate the entropy of a hom-shift \cite{PS-TCS2018}. (A hom-shift, roughly speaking, is a $G$-SFT which is isotropic and symmetric; alternatively, a hom-shift is determined by one rule in each direction. For instance, a $d$-dimensional golden mean shift is a hom-shift. The interested reader is referred to \cite{CM-PJM2018}.) Furthermore, there is an infinite series expression for the entropy provided $|\mathcal{A}| = d = 2$ \cite{BC-2017b}.

This paper considers the topological degree of $G$-SFTs for the case where $G$ is a finitely generated nonabelian monoid. The topological degree of a $G$-shift reflects the idea of entropy dimension. More specifically, the topological degree of a $G$-shift having positive topological entropy is $\ln d$, where $d$ is the number of generators of $G$ and $G$ is a free monoid (cf.~\cite{BC-2017, BC-N2017, BC-TAMS2017}). In other words, the investigation of topological degree relates to discovering zero entropy systems. The importance of zero entropy systems has been revealed recently; many $\mathbb{Z}^d$-actions with zero entropy exhibit diverse complexities. See \cite{Carvalho-PM1997,CL-JSP2010,DHP-TAMS2011,DHK-2018} and the references therein for more details. This elucidation extends the computation of topological degree of $G$-SFT to the case where $G$ is a monoid with finite representation (Theorem \ref{thm:degree-essential-case}) or $G$ has finite \emph{free-followers} (defined in \eqref{eq:follower-set-definition}, see Section \ref{sec:finite-follower-case}). Notably, a finitely generated free monoid has finite representation, and a monoid with finite representation has finite free-followers.

On the other hand, Ban \emph{et al.}~\cite{BCH-2017} reveal that, if $G$ is a free monoid with $d$ generators, the degree spectrum of $G$-SFTs is a finite subset of Perron numbers less than or equal to $d$. Theorems \ref{thm:degree-general-case-2symbol} and \ref{thm:degree-general-case-ksymbol} elaborate that the topological degree of a $G$-SFT relates to the maximal spectral radius of a collection of integral matrices which are constrained by the structure of the Cayley graph of $G$. Meanwhile, the necessary and sufficient conditions for a $G$-SFT having full topological degree are also addressed, which provide a criterion for determining whether a $G$-SFT has zero entropy.

The introduction ends with a summary of the remainder of the paper. Whenever $G$ is a monoid such that a matrix $A$ determines the relationships among the generators of $G$ and $G$ has \emph{finite representation} (see Section \ref{sec:definition}), the coefficients of the characteristic polynomial of $A$ relate to the number of children of the Cayley graph of $G$ (Theorem \ref{thm:char-poly-A-xi}). After revealing an algorithm for the computation of the topological degree (Theorem \ref{thm:degree-algorithm}), the degree spectrum (Theorems \ref{thm:degree-general-case-2symbol} and \ref{thm:degree-general-case-ksymbol}) extends the previous result under the hypothesis that $G$ is free (cf.~\cite{BCH-2017}). Furthermore, Section \ref{sec:finite-follower-case} extends the algorithm to the case where $G$ has finite free-followers.

\section{Definition and Notation} \label{sec:definition}

Although most results in this investigation extend to groups with finite representation, the present paper focuses on shift spaces on monoids for clarity. Let $d$ be a positive integer. A \emph{semigroup} is a set $G = \langle \Sigma | R \rangle$ together with a binary operation which is closed and associative, where $\Sigma = \{s_1, \ldots, s_d\}$ is the set of generators and $R$ is a set of equivalences which describe the relationships among the generators. A \emph{monoid} is a semigroup with an identity element $e$.

Given a finite set of generators $\Sigma = \{s_1, s_2, \ldots, s_d\}$ and a $d \times d$ binary matrix $A$. A monoid $G$ of the form $G = \langle \Sigma | R_A \rangle$ means that
$$
s_i s_j = s_i \quad \text{if and only if} \quad A(i, j) = 0.
$$
Alternatively, $s_i$ is a right (resp.~left) free generator if and only if $A(i, j) = 1$ ($A(j, i) = 1$) for $1 \leq j \leq d$. Let $\Sigma_R$ (resp.~$\Sigma_L$) denote the set of right (resp.~left) free generators of $G$. For each $g \in G$, the \emph{length} $|g|$ indicates the number of generators used in its minimal presentation; that is,
$$
|g| = \min \{j: g = g_1 g_2 \cdots g_j, g_i \in \Sigma \text{ for } 1 \leq i \leq j\}.
$$
Suppose that $C = (V, E)$ is the Cayley graph of $G$. Define a subgraph $F = (V_F, E_F) \subseteq C$ as follows.
\begin{enumerate}[\bf (i)]
\item $g = g_1 \cdots g_n \in V_F$ if $g_n$ is the unique right free generator in $g$;
\item if $g_1 \cdots g_n \in V_F$, then $g_1 \cdots g_j \in V_F$ for each $j \leq n$;
\item $(g, g') \in E_F$ if and only if $(g, g') \in E$ and $g, g' \in V_F$.
\end{enumerate}
A monoid $G$ has a \emph{finite representation} if $F$ is a finite graph. For the rest of this paper, $G = \langle \Sigma | R_A \rangle$ denotes a monoid with a finite representation unless otherwise stated. See Example \ref{eg:G-A3x3-semigroup}.

\begin{figure}
\centering
\begin{subfigure}{\textwidth}
\begin{center}
\includegraphics[scale=0.8,page=1]{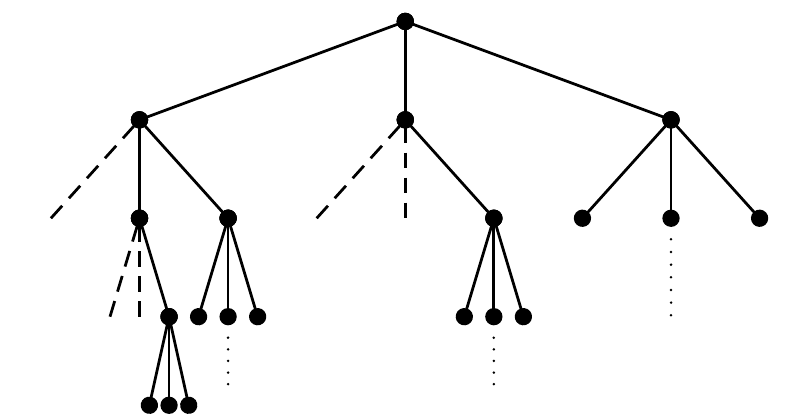}
\end{center}
\caption{Part of Cayley graph of monoid $G$ in Example \ref{eg:G-A3x3-semigroup}.}
\label{fig:G-A3x3-semigroup-C}
\end{subfigure}

\begin{subfigure}{\textwidth}
\begin{center}
\includegraphics[scale=0.8,page=2]{DegreeFiniteMonoid-20180628-pics-arxiv}
\end{center}
\caption{Finite representation of the Cayley graph $C$ of Example \ref{eg:G-A3x3-semigroup}.}\label{fig:G-A3x3-semigroup-CF}
\end{subfigure}

\caption{The Cayley graph of monoid $G$ in Example \ref{eg:G-A3x3-semigroup} has a finite representation. The generators $s_1, s_2, s_3$ satisfy the equivalences $s_1^2 = s_1$ and $s_2 s_1 = s_2^2 = s_2$. A pseudo-identity $e$ makes the Cayley graph of $G$ strongly connected.}\label{fig:G-A3x3-semigroup}
\end{figure}

\begin{example} \label{eg:G-A3x3-semigroup}
Let $d = 3$ and let
$$
A =
\begin{pmatrix}
0 &1 &1 \\
0 &0 &1 \\
1 &1 &1
\end{pmatrix}.
$$
The generators $s_1, s_2, s_3$ satisfy the equivalences determined by $A$; that is,
$$
s_1^2 = s_1, s_2 s_1 =  s_2, \text{ and } s_2^2 = s_2.
$$
It follows that $G$ has a finite representation, sees Figure \ref{fig:G-A3x3-semigroup} for the Cayley graph of $G$ together with a pseudo-identity and the graph of its finite representation.
\end{example}

Let $\mathcal{A} = \{1, 2, \ldots, k\}$ be a finite alphabet, where $k$ is a natural number greater than one. A \emph{labeled tree} is a function $t: G \to \mathcal{A}$. For each $g \in G$, $t_g = t(g)$ denotes the label attached to the vertex $g$ of the Cayley graph of $G$. The \emph{full shift} $\mathcal{A}^G$ collects the labeled trees, and the \emph{shift map} $\sigma: G \times \mathcal{A}^G \to \mathcal{A}^G$ is defined as $(\sigma_g t)_{g'} = t_{gg'}$ for $g, g' \in G$.

For each $n \geq 0$, let $\Delta_n = \{g \in G: |g| \leq n\}$ denote the initial $n$-subgraph of the Cayley graph. An \emph{$n$-block} is a function $\tau: \Delta_n \to \mathcal{A}$. A labeled tree $t$ \emph{accepts} an $n$-block $\tau$ if there exists $g \in G$ such that $t_{g g'} = \tau_{g'}$ for all $g' \in \Delta_n$; otherwise, $\tau$ is a \emph{forbidden block} of $t$ (or $t$ \emph{avoids} $\tau$). A \emph{$G$-shift space} is a set $X \subseteq \mathcal{A}^G$ of all labeled trees which avoid all of a certain set of forbidden blocks.

\section{Characterization of Finite Representation} \label{sec:finite-representation-characteristic-polynomial}

For each $n \in \mathbb{N}$, let
$$
\xi_n = \# \{g \in G: |g| = n, g_n \text{ is the only right free generator}\}
$$
be the number of $n$-words which contain exactly one right free generator and end in it. Theorem \ref{thm:char-poly-A-xi} reveals $\{\xi_n\}$ plays an important role in the characteristic polynomial of $A$.

\begin{theorem}\label{thm:char-poly-A-xi}
Suppose $G = \langle \Sigma | R_A \rangle$ is a monoid determined by a binary matrix $A$. Then the characteristic polynomial of $A$ is
\begin{equation}\label{eq:char-poly-A-xi}
f(\lambda) = \lambda^d - \sum_{i=1}^d \xi_i \lambda^{d-i}.
\end{equation}
\end{theorem}

Before proving Theorem \ref{thm:char-poly-A-xi}, it is essential to characterize the structure of the Cayley graph of $G$. Let
$$
P_n = \{g \in G: |g| = n+1 \text{ and } g_1 = g_{n+1}\}
$$
and
$$
\Xi _n = \{g \in P_n: g_n \text{ is the only right free generator}\}
$$
be the sets of periodic $(n+1)$-words and periodic $(n+1)$-words whose second last symbol is the one and only right free generator, respectively. It follows immediately that $|P_n| = \mathrm{tr}(A^n)$ and $|\Xi _n| = \xi_n$.

\begin{lemma}\label{lem:periodic-word-contain-free-generator}
For each $(n+1)$-word $g = g_1 g_2 \cdots g_n g_1 \in P_n$, there exists $1 \leq i \leq n$ such that $g_i$ is a right free generator.
\end{lemma}
\begin{proof} 
Suppose not, it comes immediately that $(g_1 \cdots g_n)^m \in G$ is a vertex of $F$ for all $m \in \mathbb{N}$, which contradicts to $|V_F| < \infty$. The proof is complete.
\end{proof}

\begin{lemma}
For each positive integer $n > d$, $\xi_n = 0$. That is, every $(n+1)$-word contains at least one right free generator.
\end{lemma}
\begin{proof} 
Suppose there exists $n > d$ and $g = g_1 \cdots g_n \in G$ such that $g_n$ is the only free generator. The pigeonhole  principle asserts that $g_i = g_j$ for some $1 \leq i < j \leq d+1$. Lemma \ref{lem:periodic-word-contain-free-generator} demonstrates that there exists $i \leq \ell \leq j$ such that $g_{\ell}$ is a right free generator, which is a contradiction. This derives the desired result.
\end{proof}

Let
$$
T(\Xi _n) = \{ u_i \cdots u_n u_1 \cdots u_i: i = 1, \ldots, n, u \in \Xi_n \}
$$
collect the translation of all elements of $\Xi_n$. Observe that $|T(\Xi_n)| = n\xi_n$. Let
$$
L(P_m, \Xi_n) = \{u_1 \cdots u_s v_1 \cdots v_n u_{s+1} \cdots u_{m+1}: s = \min_i \{i: u_i \in S_R\}, u \in P_m, v \in \Xi_n \},
$$
where $S_R$ is the set of right free generators. That is, $L(P_m, \Xi_n)$ consists of words obtained by inserting the initial $n$-subword of every $v = v_1 \cdots v_{n+1} \in \Xi_n$ in a periodic $(m+1)$-word $u$ right after the first right free generator of $u$. Obviously, $L(P_m, \Xi_n) \subseteq P_{m+n}$.

\begin{lemma}\label{lem:1st2nd-generator-distance}
Suppose $x = x_1 \cdots x_{n+1} \in P_n$ contains at least two right free generators. Let $x_s$ and $x_r$ be the first and second free generators, respectively. Then
$$
r - s = l \quad \text{if and only if} \quad x \in L(P_{n-l}, \Xi_l)
$$
\end{lemma}
\begin{proof}
If $x \in L(P_{n-l}, \Xi_l)$, then $x = u_1 \cdots u_s v_1 \cdots v_l u_{s+1} \cdots u_{n-l+1}$ for some $u \in P_{n-l}, v \in \Xi_l$, where $s = \min\{i: u_i \in S_R\}$. Since $v \in \Xi_l$, $x_{s + l} = v_l$ is the second right free generator in $x$. This concludes that $r - s = ( s + l ) - s = l$.

For each $x = x_1 \cdots x_{n+1} \in P_n$ which contains at least two right free generators, let $u = x_1 \cdots x_s x_{s+l+1} \cdots x_{n+1}$ and let $v = x_{s+1} \cdots x_{s+l} x_{s+1}$. Since $u_s = x_s \in S_R $, $x_s x_{s+l-1}$ is a two-word. Furthermore, $x \in P_n$ and $x_1 = x_{n+1}$ indicates that $u_1 = u_{n-l+1}$. Alternatively, $u = u_1 \cdots u_{n-l+1} \in P_{n-l}$. Similarly, $v_l = x_{s+l} \in S_R$ shows that $x_{s+l} x_{s+1}$ is also a two-word. The fact of $x_r = x_{s + l}$ being the second free generator elaborates that $v_1, \ldots, v_{l+1} \notin S_R$, $v_l\in S_R$, and $v_{l+1}= v_1$. Hence, $v \in \Xi_l$. Conclusively, $x \in L(P_{n-l}, \Xi_l)$. This completes the proof.
\end{proof}

Lemma \ref{lem:1st2nd-generator-distance} illustrates $L(P_{n-l}, \Xi_l) \bigcap L(P_{n-m}, \Xi_m) = \varnothing$ if and only if $l \neq m$. Proposition \ref{prop:T-L-partition-Pn}, additionally, reveals that the translation and insertion of $\Xi_n$ form a partition of periodic words.

\begin{proposition}\label{prop:T-L-partition-Pn}
For each $n \in \mathbb{N}$, $\{L(P_{n-i}, \Xi_i)\}_{i=1}^{n-1} \bigcup \{T(\Xi_n)\}$ forms a partition of $P_n$.
\end{proposition}
\begin{proof}
Obviously, $L(P_{n-i}, \Xi_i) \bigcap T(\Xi_n) = \varnothing$ for $1 \leq i \leq n-1$ since every element of $T(\Xi_n)$ accepts exactly one generator while $L(P_{n-i}, \Xi_i)$ consists of words which contain at least two free generators. The desired result comes immediately from
$$
P_n = \bigcup_{i=1}^{n-1} L(P_{n-i}, \Xi_i) \bigcup T(\Xi_n).
$$
Indeed, the definitions of $L(P_{n-i}, \Xi_i)$ and $T(\Xi_n)$ indicate that
$$
\bigcup_{i=1}^{n-1} L(P_{n-i}, \Xi_i) \bigcup T(\Xi_n) \subseteq P_n.
$$
For each $x \in P_n$, $x \in T(\Xi_n)$ if $x$ has exactly one free generator. Otherwise, $x$ has $x_s$ and $x_r$ as its first two free generators for some $s < r$. Let $l = r - s$. Lemma \ref{lem:1st2nd-generator-distance} shows that $x \in L(P_{n-l}, \Xi_l)$. The proof is complete.
\end{proof}

\begin{example}
Continue with Example \ref{eg:G-A3x3-semigroup}, recall that
\begin{equation*}
A=
\begin{pmatrix}
0 &1 &1 \\
0 &0 &1 \\
1 &1 &1
\end{pmatrix}
\end{equation*}
and $G$ has only one right free generator $s_3$. Then $\Xi_1 = \{s_3 s_3\} = T(\Xi_1) = P_1$. Since $\Xi_2$ consists of words of the form $u_1 s_3 u_1$,
$$
\Xi_2 = \{s_1 s_3 s_1, s_2 s_3 s_2\} \quad \text{and} \quad T(\Xi_2) = \{s_1 s_3 s_1, s_3 s_1 s_3, s_2 s_3 s_2, s_3 s_2 s_3\}.
$$
As defined above, $L(P_1, \Xi_1) = \{s_3 s_3 s_3\}$ collects the words obtained by inserting the first word of $\Xi_1$ in each word of $P_1$ right after the first right free generator. It follows that
$$
P_2 = \{s_1 s_3 s_1, s_2 s_3 s_2, s_3 s_1 s_3, s_3 s_2 s_3, s_3 s_3 s_3\} = T(\Xi_2) \bigcup L(P_1, \Xi_1).
$$

Similarly, $\Xi_3 = \{s_1 s_2 s_3 s_1\}$ and $T(\Xi_3) = \{s_1 s_2 s_3 s_1, s_2 s_3 s_1 s_2, s_3 s_1 s_2 s_3\}$.
\begin{align*}
L(P_2, \Xi_1) &= \{s_1 s_3 s_3 s_1, s_2 s_3 s_3 s_2, s_3 s_3 s_1 s_3, s_3 s_3 s_2 s_3, s_3 s_3 s_3 s_3\}, \\
L(P_1,\Xi_2) &= \{s_3 s_1 s_3 s_3, s_2 s_3 s_3 s_3\}.
\end{align*}
Then
\begin{align*}
P_3 &= \{s_1 s_2 s_3 s_1, s_1 s_3 s_3 s_1, s_2 s_3 s_1 s_2, s_2 s_3 s_3 s_2, s_3 s_1 s_2 s_3, \\
 & \qquad s_3 s_1 s_3 s_3, s_3 s_2 s_3 s_3, s_3 s_3 s_1 s_3, s_3 s_3 s_2 s_3, s_3 s_3 s_3 s_3\} \\
 &= T(\Xi_3) \bigcup L(P_2, \Xi_1) \bigcup L(P_1,\Xi_2).
\end{align*}
Furthermore, $\Xi_n = \varnothing$ for $n \geq 4$.
\end{example}

For each real $n \times n$ matrix $A$, there is a recursive formula for the coefficients of the characteristic polynomial of $A$; more explicitly, $f(\lambda) = \det (A - \lambda I) = \sum_{i=0}^n b_i \lambda^{n-i}$, where
\begin{align*}
b_0 &= (-1)^n,  \quad b_1 = - (-1)^n A_1, \quad b_2 = -\dfrac{1}{2} (b_1 A_1 + (-1)^n A_2), \\
b_3 &= -\dfrac{1}{3} (b_2 A_1 + b_1 A_2 + (-1)^n A_3), \ldots \\
b_n &= -\dfrac{1}{n} (b_{n-1} A_1 + b_{n-2} A_2 + \cdots + b_1 A_{n-1} + (-1)^n A_n),
\end{align*}
and $A_i$ is the trace of $A^i$ for $1 \leq i \leq n$ (cf.~\cite[p.303-305]{ZD-1963}).

\begin{proof}[Proof of Theorem \ref{thm:char-poly-A-xi}]
Proposition \ref{prop:T-L-partition-Pn} shows that, for $n \in \mathbb{N}$,
$$
|P_n| = |T(\Xi_n)| + \sum_{i=1}^{n-1} |L(P_i, \Xi_{n-i})|;
$$
that is, $A_1 = \xi_1$ and $A_n = n \xi_n + \sum_{i=1}^{n-1} A_i \xi_{n-i}$ for $n \geq 2$. Since $\xi_n = 0$ for $n > d$, \eqref{eq:char-poly-A-xi} follows from
$$
\xi_n= \frac{1}{n}(A_n-\sum_{i=1}^{n-1}A_i\xi_{n-i}), \quad 1 \leq n \leq d,
$$
and the recursive formula of the coefficients of the characteristic polynomial of $A$.
\end{proof}

\section{Topological Degree of Shift Spaces on Monoids}\label{sec:degree-essential-case}

Suppose that $X$ is a $G$-shift space. Let $\Gamma_n^{[g]}(X)$ denote the set of $n$-blocks of $X$ rooted at $g$; that is, the support of each block of $\Gamma_n^{[g]}(X)$ is $g\Delta_n$. Let $\gamma_n^{[g]}$ denote the cardinality of $\Gamma_n^{[g]}(X)$. The \emph{topological degree} of $X$ is defined as
\begin{equation}\label{eq:defn-degree}
\mathrm{deg}(X) = \limsup_{n \to \infty} \dfrac{\ln \ln \gamma_n(X)}{n},
\end{equation}
where $\gamma_n(X) = \gamma_n^{[e]}(X)$. The rest of this paper omits the notation $X$ when it causes no confusion.

For each $a \in \mathcal{A}$, let $\Gamma_{a, n}^{[g]} \subseteq \Gamma_n^{[g]}$ consist of $n$-blocks rooted at $g$ and labeled $a$ at root. A symbol $a$ is \emph{essential}\footnote{In one-dimensional symbolic dynamical systems, a graph presentation of an SFT is called essential if there is no stranded vertex \cite{LM-1995}. In other words, every vertex has its contribution in the corresponding SFT. This paper extends the idea to the alphabet of $G$-SFTs.} if $\gamma_{a, n} = |\Gamma_{a, n}| \geq 2$ for some $n \in \mathbb{N}$; otherwise, $a$ is an \emph{inessential} symbol. Proposition \ref{prop:degree-exists-related-essential} indicates that the limit in \eqref{eq:defn-degree} exists provided $X$ is a $G$-SFT, and only the essential symbols matter for calculating the topological degree.

\begin{proposition}[See \cite{BC-2017a}]\label{prop:degree-exists-related-essential}
Suppose that $X$ is a $G$-SFT. Then the limit \eqref{eq:defn-degree} exists and
\begin{equation}
\mathrm{deg}(X) = \lim_{n \to \infty} \dfrac{\ln \sum_{i=1}^k \ln \gamma_{i, n}}{n} = \lim_{n \to \infty} \dfrac{\ln \sum_{i \in \mathcal{E}} \ln \gamma_{i, n}}{n},
\end{equation}
where $\mathcal{E} \subseteq \mathcal{A}$ denotes the set of essential symbols.
\end{proposition}

Theorem \ref{thm:degree-essential-case} reveals that, whenever every symbol is essential, the degree of $G$-shift of finite type ($G$-SFT) is the logarithm of the spectral radius of $A$ (recall that $G = \langle \Sigma | R_A \rangle$ is determined by a $d \times d$ matrix $A$, see Section \ref{sec:definition}).

\begin{theorem}\label{thm:degree-essential-case}
Suppose that $X$ is a $G$-SFT and every symbol is essential. Then $\mathrm{deg}(X) = \ln \rho_A$, where $\rho_A$ is the spectral radius of $A$.
\end{theorem}

Ban and Chang \cite{BC-2017a} reveal an algorithm for computing the degree of $G$-SFTs, where $G$ is a Fibonacci set ($G = \langle \Sigma | R_A \rangle$ with $\Sigma = \{s_1, s_2\}$ and $A = \begin{pmatrix}
1 & 1 \\
1 & 0
\end{pmatrix}$). The algorithm extends to general $G$ via analogous argument. For the compactness and self-containment of the present paper, this section rephrases main ideas and propositions of the algorithm without detailed proofs, and Example \ref{eg:hom-shift-essential-case} shows how the algorithm works. For more details about the algorithm, see \cite{BC-2017a} and the references therein.

Since $G$ has a finite representation, each $G$-SFT relates to a \emph{recurrence representation} (or \emph{system of nonlinear recurrence equations}, SNRE) of the form
\begin{align*}
\gamma_{i, n} = \sum c_\mathbf{j} \gamma_{1,n-1}^{j_{1,1}} \cdots \gamma_{k, n-1}^{j_{k,1}} \gamma_{1,n-2}^{j_{1,2}} \cdots \gamma_{k, n-2}^{j_{k,2}} \cdots  \gamma_{1,n-l}^{j_{1,l}} \cdots \gamma_{k, n-l}^{j_{k,l}}
\end{align*}
for some $l \in \mathbb{N}$, where $c_\mathbf{j} \in \mathbb{N}$, $\mathbf{j} = (j_{1,1}, \ldots, j_{k,1}, \ldots, j_{1,l}, \ldots, j_{k,l})$, and $1 \leq i \leq k$. A \emph{simple subsystem} of $X$ is of the form
\begin{align*}
\gamma_{i, n} = \gamma_{1,n-1}^{j_{1,1}} \cdots \gamma_{k, n-1}^{j_{k,1}} \gamma_{1,n-2}^{j_{1,2}} \cdots \gamma_{k, n-2}^{j_{k,2}} \cdots  \gamma_{1,n-l}^{j_{1,l}} \cdots \gamma_{k, n-l}^{j_{k,l}}
\end{align*}
for some $j_{1,1}, \ldots, j_{k,1}, \ldots, j_{1,l}, \ldots, j_{k,l}$ and $1 \leq i \leq k$. Take logarithm on the above equation and let
$$
\theta_n = ( \ln \gamma_{1,n}, \ldots, \ln \gamma_{k,n}, \ln \gamma_{1;n-1}, \ldots, \ln \gamma_{k,n-1}, \ldots, \ln \gamma_{1, n-l+1}, \ldots, \ln \gamma_{k, n-l+1})',
$$
where $v'$ refers to the transpose of $v$. Then there exists a $kl \times kl$ matrix $M$ called \emph{adjacency matrix} (of the simple subsystem) such that $\theta_n = M \theta_{n-1}$ for $n \geq l+1$. Theorem \ref{thm:degree-algorithm} reveals that the degree of $X$ relates to the maximal spectral radius among the adjacency matrices of simple subsystems of $X$.

\begin{theorem}[See \cite{BC-2017a}]\label{thm:degree-algorithm}
Suppose that $X$ is a $G$-SFT. Then
$$
\mathrm{deg}(X) = \max\{\ln \rho_M: M \text{ is the adjacency matrix of a simple subsystem of } X\}.
$$
\end{theorem}

\begin{example}\label{eg:hom-shift-essential-case}
Suppose that $X$ is a hom-shift on $G$ determined by a $k \times k$ binary matrix $T$; that is, for each labeled tree $t \in X$ and $g \in G$, a pattern $(t_g, t_{g s_i})$ is allowable if and only if $T(t_g, t_{g s_i}) = 1$. For instance, consider the case where $k = 2$ and
$T = \begin{pmatrix}
1 & 1 \\
1 & 1
\end{pmatrix}$. A hom-shift defined by $T$ is a full $G$-shift and $\mathrm{deg}(X) = \ln \rho_A$. This example shows that the above algorithm derives the desired result.

It follows from $s_3$ being a free generator that, for $i = 1, 2$, $\gamma_{i, n}^{[g]} = \gamma_{i, n}$ if $g = g' s_3$ for some $g, g' \in G$. Hence, for $i=1, 2$,
\begin{align*}
\gamma_{i,n} &= (\gamma_{1,n-1}^{[s_1]}+\gamma_{2,n-1}^{[s_1]})(\gamma_{1,n-1}^{[s_2]}+\gamma_{2,n-1}^{[s_2]})(\gamma_{1,n-1}^{[s_3]}+\gamma_{2,n-1}^{[s_3]}) \\
 &= (\gamma_{1,n-1}^{[s_1]}+\gamma_{2,n-1}^{[s_1]})(\gamma_{1,n-1}^{[s_2]}+\gamma_{2,n-1}^{[s_2]})(\gamma_{1,n-1}+\gamma_{2,n-1})
\end{align*}
Combining
\begin{align*}
\gamma_{i,n-1}^{[s_1]} &= (\gamma_{1,n-2}^{[s_1 s_2]}+\gamma_{2,n-2}^{[s_1 s_2]})(\gamma_{1,n-2}^{[s_1 s_3]}+\gamma_{2,n-2}^{[s_1 s_3]}) = (\gamma_{1,n-2}^{[s_1 s_2]}+\gamma_{2,n-2}^{[s_1 s_2]})(\gamma_{1,n-2} + \gamma_{2,n-2}) \\
\gamma_{i,n-1}^{[s_2]} &= \gamma_{1,n-2}^{[s_2 s_3]}+\gamma_{2,n-2}^{[s_2 s_3]} = \gamma_{1,n-2} + \gamma_{2,n-2}
\end{align*}%
with
\begin{align*}
\gamma_{i,n-1}^{[s_1 s_2]} = \gamma_{1,n-3}^{[s_1 s_2 s_3]} + \gamma_{2,n-3}^{[s_1 s_2 s_3]} = \gamma_{1,n-3} + \gamma_{2,n-3}
\end{align*}
derives that
\begin{align*}
\gamma_{i,n} &=
(4\gamma_{1,n-3}\gamma_{1,n-2}+4\gamma_{2,n-3}\gamma_{1,n-2}+4\gamma_{1,n-3}\gamma_{2,n-2}+4\gamma_{2,n-3}\gamma_{2,n-2}) \cdot \\
 & \qquad (2\gamma_{1,n-2}+2\gamma_{2,n-2}) (\gamma_{1,n-1}+\gamma_{2,n-1})
\end{align*}
Let
$$
\theta_n =  (\ln \gamma_{1,n},\ln \gamma_{2,n},\ln \gamma_{1,n-1},\ln \gamma_{2,n-1},\ln \gamma_{1,n-2},\ln \gamma_{2,n-2})'.
$$
For every simple subsystem of $X$, the corresponding adjacency matrix is of the form
\begin{equation*}
M=
\begin{pmatrix}
B_{1}&B_{2}&B_{3}\\ 
I&0&0\\ 
0&I&0%
\end{pmatrix},
\end{equation*}
where $B_l$ is a $2 \times 2$ matrix satisfies $\sum_{q = 1}^2 B_l(p, q) = \xi_l$ for all $l = 1,2,3$, $p = 1, 2$. That is, $\theta_n = M \theta_{n-1}$ for $n \geq 3$. Let
$$
v = (\rho_A^2, \rho_A^2, \rho_A, \rho_A, 1, 1)'.
$$
Observe that $Mv = \rho_A v$. Perron-Frobenius Theorem demonstrates that $\rho_A$ is also the spectral radius of $M$. In other words, $\mathrm{deg}(X) = \ln \rho_A$.
\end{example}

\begin{proof}[Proof of Theorem \ref{thm:degree-essential-case}]
The proof focuses on the case where $X$ is a $G$-SFT determined by $k \times k$ binary matrices $A_1, \ldots, A_d$ for clarification, the demonstration of the general case is analogous. In this case, for each labeled tree $t \in X$ and $g \in G$, $(t_g, t_{g s_l})$ is allowable if and only if $A_l (t_g, t_{g s_l}) = 1$ for $1 \leq l \leq d$.

Write $A_l = (a_{l; i_1, i_2})$ for $1 \leq l \leq d, 1 \leq i_1, i_2 \leq k$. Since $\gamma_{j, n}^{[g s_l]} = \gamma_{j, n}$ for all $1 \leq j \leq k, n \in \mathbb{N}, g \in G$ provided $s_l$ is a free generator, for $1 \leq i \leq k$,
\begin{align*}
\gamma_{i,n} &= \prod_{s_l \in \Sigma} (\sum_{j_1=1}^k a_{s_l;i,j_1}\gamma_{{j_1},n-1}^{[s_l]}) \\
 &= \prod_{s_l \notin S_R} (\sum_{j_1=1}^k a_{s_l;i,j_1}\gamma_{{j_1},n-1}^{[s_l]})\prod_{s_l\in S_R} (\sum_{j_1=1}^k a_{s_l;i,j_1} \gamma_{{j_1},n-1}^{[s_l]}) \\
 &= \prod_{s_l \notin S_R} (\sum_{j_1=1}^k a_{s_l;i,j_1}\gamma_{{j_1},n-1}^{[s_l]})\prod_{s_l\in S_R} (\sum_{j_1=1}^k a_{s_l;i,j_1} \gamma_{{j_1},n-1})
\end{align*}
Observe that $f_1 = \prod_{s_l\in S_R} (\sum_{j_1=1}^k a_{s_l;i,j_1} \gamma_{{j_1},n-1})$ is a polynomial of degree $\xi_1$ over $\gamma_{1,n-1}, \ldots,\gamma_{k,n-1}$.

Similarly, for each $s_l$ which is not a free generator,
\begin{align*}
\gamma_{j_1,n-1}^{[s_l]} = \prod_{s_l s_m \in G} (\sum_{j_2=1}^k a_{s_m; j_1, j_2} \gamma_{{j_2},n-2}^{[s_l s_m]})
\end{align*}%
infers that
\begin{align*}
\gamma_{i,n} = f_1 \cdot \prod_{s_l s_m \in G, s_l \notin S_R} (\sum_{j_1,j_2=1}^k a_{s_l;i,j_1} a_{s_m;j_1,j_2}\gamma_{j_2,n-2}^{[s_l s_m]}).
\end{align*}
Let
\begin{align*}
f_2 &= \prod_{s_l s_m \in G, s_l \notin S_R, s_m \in S_R} (\sum_{j_1,j_2=1}^k a_{s_l;i,j_1} a_{s_m;j_1,j_2}\gamma_{j_2,n-2}^{[s_l s_m]}) \\
 &= \prod_{s_l s_m \in G, s_l \notin S_R, s_m \in S_R} (\sum_{j_1,j_2=1}^k a_{s_l;i,j_1} a_{s_m;j_1,j_2}\gamma_{j_2,n-2})
\end{align*}
Then $f_2$ is a polynomial of degree $\xi_2$. Repeating the same process decompose $\gamma_{i,n} = f_1 f_2 \cdots f_{\ell}$, where $\ell = \max\{j: \xi_j \neq 0\} \leq d$, and $f_j$ is a polynomial of degree $\xi_j$ over $\gamma_{1,n-j}, \ldots, \gamma_{k,n-j}$ for $1 \leq j \leq \ell$.

Let
$$
\theta_n= (\ln \gamma_{1,n},\cdots,\ln \gamma_{k,n},\ln \gamma_{1,n-1},\cdots,\ln \gamma_{k,n-1},\cdots,\ln \gamma_{1,n-d+1},\cdots,\ln \gamma_{k,n-d+1})'.
$$
For each simple subsystem of $X$, there exists
\begin{equation*}
M=
\begin{pmatrix}
B_1& B_2& B_3& \cdots & B_{\ell} \\ 
I& 0 & \cdots & \cdots &0 \\ 
0&I & 0 &  \cdots  & 0\\ 
\vdots && \ddots& & \vdots  \\ 
0 & \cdots  & 0 &I&0%
\end{pmatrix},
\end{equation*}
where $B_j$ is a $k \times k$ nonnegative integral matrix satisfies $\sum_{q=1}^k B_j (p, q) = \xi_i$ for all $1 \leq j \leq \ell, 1 \leq p \leq k$, such that $M$ is the corresponding adjacency matrix (note that $\xi_j = 0$ for $j > \ell$). That is, $\theta_n = M \theta_{n-1}$ is the designated simple subsystem. Let $v = (\rho_A^{\ell-1} \cdots \rho_A 1)' \otimes \mathbf{1}_k$, where $\otimes$ is the Kronecker product and $\mathbf{1}_k \in \mathbb{R}^k$ is the vector consisting of $1$'s. It comes immediately that $M v = \rho_A v$. Perron-Frobenius Theorem infers that $\rho_A$ is also the spectral radius of $M$. Hence, $\mathrm{deg}(X) = \ln\rho_A$.

This completes the proof.
\end{proof}

\begin{remark}
For the general cases, Proposition \ref{prop:degree-exists-related-essential} demonstrates that Theorem \ref{thm:degree-algorithm} holds if the rows and columns of matrix $M$ indexed by inessential symbols are eliminated.
\end{remark}

\section{Degree Spectrum of $G$-SFTs}\label{sec:degree-spectrum}

Theorem \ref{thm:degree-essential-case} reveals that the degree of $G$-SFTs is $\ln \rho_A$ whenever every symbol is essential. This section extends to the general case and gives the complete characterization of degree spectrum of $G$-SFTs.

Let $\mathbb{Z}_+$ be the set of nonnegative integers. For $\mathbf{m}, \mathbf{n} \in \mathbb{Z}_+^d$, define $\mathbf{m} \preceq \mathbf{n}$ if $m_i \leq n_i$ for $1 \leq i \leq d$, and $\mathbf{m} \prec \mathbf{n}$ if $\mathbf{m} \preceq \mathbf{n}$ and $\mathbf{m} \neq \mathbf{n}$. Theorem \ref{thm:degree-general-case-2symbol} characterizes the degree spectrum of $G$-SFTs for the case where $k = 2$.

\begin{theorem}\label{thm:degree-general-case-2symbol}
Suppose that $k = 2$. Let $\xi = (\xi_1, \ldots, \xi_d)$. The degree spectrum of $G$-SFTs is
\begin{align*}
H = \{\ln \lambda: \lambda = \max \{x: x^d - \sum_{i=1}^d \alpha_i x^{d-i} = 0\} \text{ for some } \alpha \in \mathbb{Z}_+^d, \alpha \preceq \xi \}.
\end{align*}
\end{theorem}
\begin{proof}
Obviously, two inessential symbols infers that the degree is $0$; Theorem \ref{thm:degree-essential-case} indicates the degree is $\ln \rho_A$ and $\rho_A = \max \{x: x^d - \sum_{i=1}^d \beta_i x^{d-i} = 0\}$ if every symbol is essential. It suffices to consider the case where $1 \in \mathcal{A}$ is essential and $2 \in \mathcal{A}$ is inessential.

Without loss of generality, assume that $\xi_i > 0$ for $1 \leq i \leq d$. Similar to the discussion in Example \ref{eg:hom-shift-essential-case}, write $\gamma_{1,n} = f_1 f_2 \cdots f_d$, where
$$
f_1 = \prod_{u_1 \in S_R} (\sum_{j_1=1}^2 a_{u_1;1,j_1} \gamma_{j_1,n-1})
$$
and
$$
f_i = \prod_{u_1\cdots u_i \in G, u_1, \ldots, u_{i-1} \notin S_R, u_i \in S_R} (\sum_{j_1, \ldots, j_i =1}^2 a_{u_1;1,j_1} a_{u_2;j_1,j_2}\cdots a_{u_i; j_{i-1}, j_i} \gamma_{{j_i}, n-i})
$$
for $2 \leq i \leq d$, and $f_i$ is a polynomial of degree $\xi_i$. Hence, every simple subsystem of $X$ is of the form 
\begin{align*}
\gamma_{1,n} &= \gamma_{1,n-1}^{\eta_1}\gamma_{2,n-1}^{\tau_1}\gamma_{1,n-2}^{\eta_2}\gamma_{2,n-2}^{\tau_2}\cdots\gamma_{1,n-d}^{\eta_d}\gamma_{2,n-d}^{\tau_d}, \\
\gamma_{2,n} &= \gamma_{2,n-1}^{\xi_1} \gamma_{2,n-2}^{\xi_2} \cdots \gamma_{2,n-d}^{\xi_d},
\end{align*}%
where $\eta_i+\tau_i = \xi_i$ for $1 \leq i \leq d$.

Let $\theta_n= (\ln \gamma_{1,n},\ln \gamma_{2,n},\ln \gamma_{1,n-1},\ln \gamma_{2,n-1},\cdots,\ln \gamma_{1,n-d+1},\ln \gamma_{2,n-d+1})'$, and let
\begin{equation*}
M=
\begin{pmatrix}
\eta_1& \tau_1& \eta_2   & \tau_2   &\cdots & \eta_d   & \tau_d \\ 
0  & \xi_1  &   0   &   \xi_2   &\cdots &0   &\xi_d \\ 
1  &    &       &       &       & 0     &0 \\ 
   & 1  &       &       &       &\vdots & \vdots\\
   &    &\ddots &       &       &\vdots &\vdots\\
   &    &       & \ddots&       &\vdots &\vdots \\ 
   &    &       &       &   1   &0      &0%
\end{pmatrix}.
\end{equation*}
Then the simple subsystem is $\theta_n = M \theta_{n-1}$. Since $2$ is inessential, the degree of such a simple subsystem is $\ln \lambda$, where $\lambda$ is the spectral radius of
\begin{equation*}
\overline{M} =
\begin{pmatrix}
\eta_1& \eta_2  &\cdots& \eta_d   \\ 
1  &      &      &0 \\ 
   &\ddots&      &\vdots\\
   &      &   1  &0%
\end{pmatrix}.
\end{equation*}
A straightforward examination elaborates that $\lambda = \max \{x: x^d - \sum_{i=1}^d \eta_i x^{d-i} = 0\}$. This derives
$$
H \subseteq \{\ln \lambda: \lambda = \max \{x: x^d - \sum_{i=1}^d \alpha_i x^{d-i} = 0\} \text{ for some } \alpha \in \mathbb{Z}_+^d, \alpha \preceq \xi \}.
$$

To show that, for each $\alpha \in \mathbb{Z}_+^d$ satisfying $\alpha \preceq \xi$, there exists a $G$-SFT such that $\mathrm{deg}{X} = \ln \lambda$ with $\lambda = \max \{x: x^d - \sum_{i=1}^d \alpha_i x^{d-i} = 0\}$, construct a one-step $G$-SFT as follows. Without loss of generality, assume that $\xi_i > 0$ for $i \leq d$. The symbol $2$ is inessential in the following construction, thus it suffices to mention where to label $1$.

For $n \in \mathbb{N}$, let $S_1 = \{e\}$ and, for $n \geq 2$, let
$$
S_n = \{g = g_1 \cdots g_{n-1}: g s \in G \text{ for some } s \in S_R, g_i \notin S_R \text{ for } 1 \leq i \leq n-1\}.
$$
Observe that $S_n = \varnothing$ if and only if $n > d$ (under the assumption that $\xi_n = 0$ if and only if $n > d$). Let
$$
\overline{S}_n = \{g s: g \in S_n, s \in \Sigma, g s \in G\}.
$$
Then $\bigcup_{i=1}^d \overline{S}_n$ is the set of supports of two-blocks of $X$ up to shift. For $n = 1$, let $B_1 \subseteq \overline{S}_1^{\mathcal{A}}$ consists of $1$-blocks $\phi$ which satisfy $\phi_g = 1$ if and only if
$$
g \in S_1 \bigcup \Sigma \setminus S_R \quad \text{and} \quad |\{g \in S_R: \phi_g = 1\}| = \alpha_1.
$$
In other words, each pattern of $B_1$ labels $1$ at, except from the root and non-free generators, arbitrary $\alpha_1$ free generators. This makes $\max\{p: \gamma_{1,n-1}^p | \gamma_{1, n}\} = \alpha_1$.

Analogously, let $B_2 \subseteq \overline{S}_2^{\mathcal{A}}$ consists of $1$-blocks $\phi$ which satisfy $\phi_g = 1$ if and only if
$$
g \in S_2 \bigcup \{g' s \in \overline{S}_2: s \notin S_R\} \quad \text{and} \quad |\{g' s \in \overline{S}_2: s \in S_R, \phi_{g's} = 1\}| = \alpha_2.
$$
Then $\max\{p: \gamma_{1,n-2}^p | \gamma_{1, n}\} = \alpha_2$. Repeating the same process to construct $B_i$ for $i \leq d$ makes
$$
\max\{p: \gamma_{1,n-i}^p | \gamma_{1, n}\} = \alpha_i \quad \text{for} \quad 1 \leq i \leq d.
$$
For each subset $H \subseteq G$ such that $H$ forms the support of a one-block, observe that there exists $g \in G$ ended in free generator and $1 \leq i \leq d$ such that $H = g H'$ for some $H' \subseteq \overline{S}_i$. Then each labeled pattern of support $H$ follows the same rule as determined in $\overline{S}_i$. Notably, Such a pattern is still in $B_i$.

Therefore, every simple subsystem of $X$ generated by $B = \bigcup_{i=1}^d B_i$ is of the form
$$
\gamma_{1, n} = c \cdot \gamma_{1, n-1}^{\alpha_1} \gamma_{2, n-1}^{\beta_1} \gamma_{1, n-2}^{\alpha_2} \gamma_{2, n-2}^{\beta_2} \cdots \gamma_{1, n-d}^{\alpha_d} \gamma_{2, n-d}^{\beta_d}, \quad \gamma_{2, n} = \gamma_{2, n-1}^d,
$$
where $c$ is a constant, and $\alpha_i + \beta_i = \xi_i$ for all $i$. A straightforward examination indicates that $\mathrm{deg}{X} = \ln \lambda$ with $\lambda = \max \{x: x^d - \sum_{i=1}^d \alpha_i x^{d-i} = 0\}$.

The proof is complete.
\end{proof}

\begin{remark}
Notably, $\xi_n = 0$ for $n \geq 2$ if and only if $G$ is a free monoid. In this case, $H = \{0, \ln 2, \ldots, \ln d\}$ is revealed in \cite{BCH-2017}.
\end{remark}

Theorem \ref{thm:degree-general-case-ksymbol} extends Theorem \ref{thm:degree-general-case-2symbol} to the general case. The proof is similar, thus it is omitted.

\begin{theorem}\label{thm:degree-general-case-ksymbol}
Let $\mathcal{M}$ be the set consisting of
\begin{equation*}
M=
\begin{pmatrix}
C_1& C_2& C_3& \cdots & C_d \\ 
I& 0 & \cdots & \cdots &0 \\ 
0&I & 0 &  \cdots  & 0\\ 
\vdots && \ddots& & \vdots  \\ 
0 & \cdots  & 0 &I&0%
\end{pmatrix}
\end{equation*}%
for some $l \times l$ matrices $C_i$, $l \leq k$, satisfying $\sum_{q=1}^l C_i (p, q) \leq \xi_i$ for all $1 \leq i \leq d$, $1 \leq p \leq l$.
The degree spectrum of $G$-SFTs is
\begin{align*}
H = \{\ln \lambda: \lambda \text{ is the spectral radius of } M \in \mathcal{M} \}.
\end{align*}
\end{theorem}

Corollary \ref{cor:iff-cond-for-full-degree}, follows from the proof of Theorem \ref{thm:degree-general-case-2symbol}, elaborates a necessary and sufficient condition of a $G$-SFT achieved full degree.

\begin{corollary}\label{cor:iff-cond-for-full-degree}
Suppose that $X$ is a $G$-SFT. Then $\mathrm{deg}(X) = \ln \rho_A$ if and only if the essential symbols form a subshift on right free generators; that is, for each $s \in S_R$ and $\phi$ is a one-block with support $\mathrm{supp}(\phi) = s \Sigma$, $\phi_g$ is essential for $g \in \mathrm{supp}(\phi)$.
\end{corollary}
\begin{proof}
It suffices to consider the case where $k = 2$ since the demonstration of the general case is analogous but more complicated. Recall that, in the proof of Theorem \ref{thm:degree-general-case-2symbol}, every simple subsystem of $X$ is of the form
\begin{align*}
\gamma_{1,n} &= \gamma_{1,n-1}^{\eta_1}\gamma_{2,n-1}^{\tau_1}\gamma_{1,n-2}^{\eta_2}\gamma_{2,n-2}^{\tau_2}\cdots\gamma_{1,n-d}^{\eta_d}\gamma_{2,n-d}^{\tau_d}, \\
\gamma_{2,n} &= \gamma_{1,n-1}^{\delta_1}\gamma_{2,n-1}^{\iota_1}\gamma_{1,n-2}^{\delta_2}\gamma_{2,n-2}^{\iota_2}\cdots\gamma_{1,n-d}^{\delta_d}\gamma_{2,n-d}^{\iota_d},
\end{align*}
where $\eta_i + \tau_i = \delta_i + \iota_i = \xi_i$ for $1 \leq i \leq d$. In other words, $\mathrm{deg}(X) = \ln \lambda$, where $\lambda$ is the spectral radius of one of the following matrix, which depends on the essential of symbols.
\begin{align*}
M_1 &=
\begin{pmatrix}
\eta_1& \tau_1& \eta_2   & \tau_2   &\cdots & \eta_d   & \tau_d \\ 
\delta_1  & \iota_1  &   \delta_2   &   \iota_2   &\cdots &\delta_d   &\iota_d \\ 
1  &    &       &       &       & 0     &0 \\ 
   & 1  &       &       &       &\vdots & \vdots\\
   &    &\ddots &       &       &\vdots &\vdots\\
   &    &       & \ddots&       &\vdots &\vdots \\ 
   &    &       &       &   1   &0      &0%
\end{pmatrix}, \\
M_2 &=
\begin{pmatrix}
\eta_1& \eta_2  &\cdots& \eta_d   \\ 
1  &      &      &0 \\ 
   &\ddots&      &\vdots\\
   &      &   1  &0
\end{pmatrix}, \quad M_3 =
\begin{pmatrix}
\iota_1& \iota_2  &\cdots& \iota_d   \\ 
1  &      &      &0 \\ 
   &\ddots&      &\vdots\\
   &      &   1  &0
\end{pmatrix}.
\end{align*}
It follows that $\lambda = \rho_A$ if and only if exactly one of the following three conditions holds.
\begin{enumerate}[\bf a.]
\item (Case $M_1$) Two symbols are essential.
\item (Case $M_2$) Symbol $1$ is essential and $\eta_i = \xi_i$ for $1 \leq i \leq d$.
\item (Case $M_3$) Symbol $2$ is essential and $\iota_i = \xi_i$ for $1 \leq i \leq d$.
\end{enumerate}
This completes the proof.
\end{proof}

\section{Groups with Finite Free-Followers}\label{sec:finite-follower-case}

Suppose that $G$ is a monoid. For each $g \in G$, defined the \emph{free-follower set} (free-follower for short) of $g$ as
\begin{equation}\label{eq:follower-set-definition}
F_g = \{g' \in G: |gg'| = |g| + |g'| \}.
\end{equation}
Set $F = \{F_g: g \in G\}$. Then $G$ \emph{has finite free-followers} if $F$ is finite. It is easily seen that every finitely generated free monoid $G$ has finite -free-followers since $F_g  = G$ for each $g \in G$. The investigation in Sections \ref{sec:degree-essential-case} and \ref{sec:degree-spectrum} extends to the case where $G$ has finite free-followers via analogous elaboration. This section, rather than rephrases every result in the previous two sections, presents an example to address how to compute the degree of a $G$-SFT ($G$ has finite free-followers herein) for the compactness of the paper.

Suppose that $d = k = 2$. In this case, $\Sigma = \{s_1, s_2\}$ and $\mathcal{A} = \{1, 2\}$. Let $G = \langle \Sigma | R \rangle$ be the monoid with $R = \{s_2 s_1^{2i+1} s_2 = s_2\}_{i \geq 0}$. It follows that $G$ has finite free-followers. Indeed, let
\begin{align*}
&F_{s_1} = \{s_1, s_2, s_1^2, s_1 s_2, s_2 s_1, s_2^2, \ldots\} = G, \\
&F_{s_2} = \{s_1, s_2, s_1^2, s_2 s_1, s_2^2, s_1^3, s_1^2 s_2, \ldots\} = \{s_1^n\}_{n \geq 1} \bigcup \{s_1^{2i} s_2 g: g = s_1^n, s_1^{2j} s_2^n, i, j, n \geq 0 \}, \\
&F_{s_2 s_1} = \{s_1, s_1^2, s_1 s_2, s_1^3, \ldots\} = \{s_1^n\}_{n \geq 1} \bigcup \{s_1^{2i+1} s_2 g: g = s_1^n, s_1^{2j} s_2^n, i, j, n \geq 0 \}.
\end{align*}
An examination indicates that, for each $g \in G$,
$$
F_g = \left\{ \begin{aligned}
&F_{s_1}, && g = s_1^n; \\
&F_{s_2}, && g \text{ ends in } s_2 s_1^{2i}, i \geq 0; \\
&F_{s_2 s_1}, && g \text{ ends in } s_2 s_1^{2i+1}, i \geq 0.
\end{aligned}\right.
$$
A straightforward examination elaborates that there is a one-to-one correspondence between the monoid $G$ and the set of finite words of one-dimensional even-shift.

Let $X$ be a hom-shift on $G$ determined by
$T = \begin{pmatrix}
1 & 1 \\
1 & 1
\end{pmatrix}$. Alternatively, $X$ is a full $G$-shift; it follows immediately that $\mathrm{deg}(X) = \ln \lambda$, where $\lambda = \frac{1 + \sqrt{5}}{2}$ satisfies $\lambda^2 - \lambda - 1 = 0$. The following shows that the algorithm in Section \ref{sec:degree-essential-case} derives the same result.

Observe that $\gamma_{i, n}^{[g]} = \gamma_{i, n}$ for $i = 1, 2$ since $F_{s_1^j} = G$ for $j \in \mathbb{N}$. For $i = 1, 2$,
\begin{align*}
\gamma_{i, n} &= (\gamma_{1, n-1}^{[s_1]} + \gamma_{2, n-1}^{[s_1]}) (\gamma_{1, n-1}^{[s_2]} + \gamma_{2, n-1}^{[s_2]}) \\
 &= (\gamma_{1, n-1} + \gamma_{2, n-1}) (\gamma_{1, n-1}^{[s_2]} + \gamma_{2, n-1}^{[s_2]})
\end{align*}
Also, $F_{s_2^2} = F_{s_2}$ and $F_{s_2 s_1^2} = F_{s_2}$ infer that
\begin{align*}
\gamma_{i, n-1}^{[s_2]} &= (\gamma_{1, n-2}^{[s_2 s_1]} + \gamma_{2, n-2}^{[s_2 s_1]}) (\gamma_{1, n-2}^{[s_2 s_2]} + \gamma_{2, n-2}^{[s_2 s_2]}) \\
 &= (\gamma_{1, n-2}^{[s_2 s_1]} + \gamma_{2, n-2}^{[s_2 s_1]}) (\gamma_{1, n-2}^{[s_2]} + \gamma_{2, n-2}^{[s_2]}) \\
\gamma_{i, n-2}^{[s_2 s_1]} &= \gamma_{1, n-3}^{[s_2 s_1^2]} + \gamma_{2, n-3}^{[s_2 s_1^2]} = \gamma_{1, n-3}^{[s_2]} + \gamma_{2, n-3}^{[s_2]}
\end{align*}
Hence, the SNRE of $X$ is
\begin{align*}
\gamma_{i, n} &=  2 (\gamma_{1, n-1} + \gamma_{2, n-1}) (\gamma_{1, n-2}^{[s_2]} + \gamma_{2, n-2}^{[s_2]}) (\gamma_{1, n-3}^{[s_2]} + \gamma_{2, n-3}^{[s_2]})
\end{align*}
for $i = 1, 2$.
Let $\theta_n = (\ln \gamma_{1, n}^{[s_2]}, \ln \gamma_{2, n}^{[s_2]}, \ln \gamma_{1, n-1}^{[s_2]}, \ln \gamma_{2, n-1}^{[s_2]})'$ and let
$$
M = \begin{pmatrix}
\eta_1 & \tau_1 & \eta_2 & \tau_2 \\
\delta_1 & \iota_1 & \delta_2 & \iota_2 \\
1 & 0 & 0 & 0 \\
0 & 1 & 0 & 0
\end{pmatrix}.
$$
Then, every simple subsystem of the invariant system $\ln \gamma_{1, n}^{[s_2]}, \ln \gamma_{2, n}^{[s_2]}$ is of the form $\theta_n = M \theta_{n-1}$ with $\eta_j + \tau_j = \delta_j + \iota_j = 1$ for $1 \leq j \leq 2$. It follows that $\ln \gamma_{i, n}^{[s_2]} \approx e^{\lambda n}$ for $i = 1, 2$ and $n$ large enough.

Furthermore, every simple subsystem of $X$ is of the form
\begin{align*}
\ln \gamma_{1, n} &\approx \eta \ln \gamma_{1, n-1} + \tau \ln \gamma_{2, n-1} + e^{(n-2) \lambda} + e^{(n-3) \lambda} \\
\ln \gamma_{2, n} &\approx \delta \ln \gamma_{1, n-1} + \iota \ln \gamma_{2, n-1} + e^{(n-2) \lambda} + e^{(n-3) \lambda}
\end{align*}
where $\eta + \tau = \delta + \iota = 1$. A straightforward examination shows that
$$
\mathrm{deg}(X) = \lim_{n \to \infty} \dfrac{\ln (\ln \gamma_{1, n} + \ln \gamma_{2, n})}{n} = \ln \lambda.
$$
This concludes the desired result.


\bibliographystyle{amsplain}
\bibliography{../../grece}

\end{document}